\documentclass[12pt,a4paper]{article}
\usepackage[latin1]{inputenc}
\usepackage{verbatim}
\usepackage{amsmath}
\usepackage{amsfonts}
\usepackage{amssymb}
\usepackage{amsthm}
\usepackage[dvips]{graphicx}
\usepackage{ae}
\usepackage[ps2pdf]{hyperref}

\numberwithin{equation}{section}
\swapnumbers

\newtheorem{theorem}[equation]{Theorem}
\newtheorem{lemma}[equation]{Lemma}

\theoremstyle{definition}
\newtheorem{definition}[equation]{\textsc{Definition}}

\theoremstyle{remark}

\newcommand{\mailto}[1]{\href{mailto:#1}{\nolinkurl{#1}}}

\def\Xint#1{\mathchoice
    {\XXint\displaystyle\textstyle{#1}}%
    {\XXint\textstyle\scriptstyle{#1}}%
    {\XXint\scriptstyle\scriptscriptstyle{#1}}%
    {\XXint\scriptscriptstyle\scriptscriptstyle{#1}}%
    \!\int}
\def\XXint#1#2#3{{\setbox0=\hbox{$#1{#2#3}{\int}$}
      \vcenter{\hbox{$#2#3$}}\kern-.5\wd0}}
\newcommand{\vint}{\Xint-}

\newcommand{\cao}{C^{\infty}_0(\Omega)}
\newcommand{\vp}{\varphi}
\newcommand{\vf}{\varphi}

\newcommand{\vep}{\varepsilon}
\renewcommand{\O}{\Omega}

\newcommand{\R}{\mathbf{R}}

\newcommand{\norm}[1]{\left\| #1 \right\|}
\newcommand{\abs}[1]{\lvert #1 \rvert}

\renewcommand{\l}{\left}
\renewcommand{\r}{\right}

\newcommand{\spt}{\operatorname{spt}}

\begin{document}
\pagenumbering{arabic}

\title{The $p$-Laplacian with respect to measures}

\author{
 Anna Tuhola-Kujanpää\thanks{
 Postal address:
 Department of Mathematics and Statistics,
 University of Jyväskylä,
 PO Box 35, 00014 University of Jyväskylä, Finland.
 Tel: +358 14 260 2789.
 URL: \url{https://www.jyu.fi/science/laitokset/maths/henkilokunta/tuhola}
 Email: \protect\mailto{aptuhola@gmail.com}}
 \and
 Harri Varpanen\thanks{
 Postal address:
 Department of Mathematics and Systems Analysis,
 Aalto University,
 PO Box 11100, 00076 Aalto, Finland.
 Tel: +358 9 4702 3069.
 URL: \url{http://math.aalto.fi/en/people/harri.varpanen} \quad
 Email: \protect\mailto{harri.varpanen@aalto.fi}}
}

\date{\today}

\maketitle

\begin{abstract}
We introduce a definition for the $p$-Laplace operator with respect to
positive and finite Borel measures that satisfy an Adams-type embedding
condition. 
\end{abstract}


\section{Introduction}
The $p$-Laplace operator $\Delta_p u = \operatorname{div}(\abs{\nabla u}^{p-2}\nabla u)$, $1< p < \infty$, is a natural nonlinear generalization of the
Laplacian in that the equation $\Delta_p u = 0$ is the
Euler-Lagrange equation for minimizing the $p$-energy functional
\[
 I(u) = \int_\O \abs{\nabla u}^p\,dx.
\]
The operator also serves as a prototype for more general classes of
nonlinear operators with $p$-growth. The elliptic theory for a wide class
of such operators is developed in \cite{HKM:book},
and the parabolic theory in \cite{DiBen:book}.

During the last few decades the study of partial differential operators on
fractals has gained popularity (see e.g. \cite{Kigami}), and indeed also
the $p$-Laplacian has been studied on Sierpinski gasket type fractals in
\cite{HerPeiStri} and \cite{StriWong}.

The purpose of our work is to introduce a related theme that extends
the linear case studied in \cite{Hu}. 
We consider a positive and finite Borel
measure $\mu$ supported on an open and bounded set $\O \subset \R^n$,
$n \ge 2$. The measure is assumed to satisfy the embedding condition
\begin{equation}\label{eq:ehto}
  \norm \vf_{L^q(\O,\mu)} \leq C \norm{\nabla \vf}_{L^p(\O)}
\end{equation}
for each $\vf \in \cao$, where
\begin{equation}\label{eq:ehto2}
1 < p \le n, \quad p < q \le \frac{np}{n-p},
\end{equation}
and $C= C(n,p,q,\O) > 0$.
This enables us to give a distributional
definition for the $p$-Laplacian with respect to the measure $\mu$ and to
study the $p$-Laplacian eigenvalue problem (see e.g. \cite{KL} and
references therein) in the measure setting.

In the case where a fractal is represented by its natural measure
our definition provides a method for studying the $p$-Laplacian with respect
to the fractal without employing any self-similar structure. It 
should, however, be emphasized that our work is a different theme in comparison
with \cite{HerPeiStri} and \cite{StriWong}.

\section{Preliminaries}
Our standing assumption throughout the
paper is \eqref{eq:ehto}.

\begin{definition}
We define the \emph{$(p,\mu)$-Laplace operator} 
$\Delta_{p,\mu}$ from $L^{q'}(\O,\mu)$ to $W_0^{1,p}(\O)$ by saying that 
\begin{equation}\label{eq:formal_mu}
  -\Delta_{p,\mu}u = f
\end{equation}
if and only if
\begin{equation}\label{eq:heikko}
  \int_\O |\nabla u|^{p-2}\nabla u \cdot \nabla \vf \, dx
= \int_\O \vf f\, d\mu
\end{equation}
for each $\vf \in C^\infty_0(\O)$. In this case, given $f \in L^{q'}(\O,\mu)$,
we also say that $u \in W^{1,p}_0(\O)$ is a \emph{solution} to $-\Delta_{p,\mu}u = f$.
\end{definition}

Here the measure $f\mu$ defines a bounded linear functional from
$W^{1,p}_0(\O)$ to $\R$, because
H\"older's inequality along with \eqref{eq:ehto} and $f \in L^{q'}(\O,\mu)$
yields
\[
\int_\O \vp f\ d\mu \le \left(\int_\O |\vp|^q \ d\mu\right)^\frac{1}{q}
\left(\int_\O |f|^{q'} \ d\mu \right)^\frac{1}{q'} \le \norm{f}_{L^{q'}(\O,\mu)}
\norm{\nabla \vp}_p < \infty.
\]
Thus a unique solution to \eqref{eq:formal_mu} always exists (see \cite[Chap. 2]{Mikkonen}), and the class of test functions can be extended to $W^{1,p}_0(\O)$
via a standard density argument (e.g. Lemma 3.11 in \cite{HKM:book}).

\begin{definition}
We say that $\lambda \in \R$ is a \emph{$(p,\mu)$-eigenvalue} if there is
a nonzero $u\in W^{1,p}_0(\O)$ such that 
\begin{equation}\label{eq}
  -\Delta_{p,\mu} u = \lambda |u|^{p-2}u.
\end{equation}
The function $u$ is called a \emph{$(p,\mu)$-eigenfunction}.
\end{definition}

Our emphasis will be on the $(p,\mu)$-eigenvalue problem, and the rest of this
paper is organized as follows.
We start in Section \ref{sec:density} by studying the relation between the
embedding \eqref{eq:ehto} and the Hausdorff dimension of the set supporting
the measure $\mu$. Section \ref{sec:eigenex} addresses positivity of
eigenvalues and existence of the first eigenfunction. Our main result,
Hölder continuity of eigenfunctions in the case $p<n$, is proved
in Section \ref{sec:continuity}.
Section \ref{sec:first} discusses the first eigenfunction, and
Section \ref{sec:n} provides a counterexample for continuity in the $p=n$ case.
Finally, Section \ref{sec:conclusions} concludes the paper.

\section{Growth of the measure} \label{sec:density}
\begin{lemma}\label{the:density}
The embedding \eqref{eq:ehto} implies
the growth condition
\begin{equation}\label{eq:mitta}
\mu(B(x_0,r))\leq C(n,p,q) r^{\frac{q}{p}(n-p)}
\end{equation}
for each $B(x_0,r) \subset \O$ such that $B(x_0,2r)\subset \O$.
Conversely, \eqref{eq:ehto} follows if \eqref{eq:mitta} holds for all $B(x_0,r) \subset \R^n$.
\end{lemma}
\begin{proof}
To show that \eqref{eq:ehto} implies \eqref{eq:mitta},
choose a
nonnegative cut-off function $\vf\in C_0^\infty(B(x_0,2r))$ such that
$\vf = 1$ in $B(x_0,r)$ and $|\nabla \vf| \le C/r$ in $B(x_0,2r)$. Then
\begin{align*}
&\mu(B(x_0,r))
= \int_{B(x_0,r)}\vf \, d\mu   
\leq \mu(B(x_0,r))^\frac{1}{q'}\l( \int_{B(x_0,2r)}|\vf|^q \, d\mu \r)^\frac{1}{q}\\
&\leq \mu(B(x_0,r))^\frac{1}{q'}\l( \int_{B(x_0,2r)}|\nabla\vf|^p \, dx
\r)^\frac{1}{p}
\leq C(n,p)\mu(B(x_0,r))^\frac{1}{q'} r^\frac{n-p}{p},
\end{align*}
i.e.
\[
\mu(B(x_0,r))^{1-\frac{1}{q'}} \leq C(n,p)r^{\frac{1}{p}(n-p)}.
\]
The converse follows readily from Adams' inequality, see e.g.
\cite[Cor. 1.93]{MalyZiemer}.
\end{proof}
Lemma \ref{the:density} motivates the condition $q \le np/(n-p)$
in \eqref{eq:ehto2}, because $\O \subset \R^n$. Indeed, the exponent $qp^{-1}(n-p)$ in \eqref{eq:mitta}
is in direct correspondence to the Hausdorff dimension of the set supporting
the measure:

\begin{theorem}\label{thm:hausdorff}
Let $E\subset \O$ be a Borel set, and denote $s = qp^{-1}(n-p)$.
\begin{itemize}
\item[(i)] If $\mathcal{H}^s(E)> 0$, then there exists a Radon measure
$\mu$ supported on $E$ that satisfies the condition \eqref{eq:ehto}.
Moreover, if $E$ is a self-similar set satisfying the open set
condition (see \eqref{eq:open} in the Appendix) and if $\dim_\mathcal{H}(E)=  s$,
then the natural self-similar measure $\mu_E$ satisfies the condition
\eqref{eq:ehto}.
\item[(ii)] If $\mu$ is such that \eqref{eq:ehto} holds
and if $\mu(\O \setminus E)= 0$, then $\dim_\mathcal{H}(E) \geq s$.
\end{itemize}
\end{theorem}
\begin{proof}
We apply Frostman's lemma (cf. \cite[Thm. 8.8]{Mattila}): For a Borel set
$B \subset \R^n$, $\mathcal{H}^s(B) > 0$ if and only if there exists a Borel
measure $\mu$ such that $\mu(B(x,r)) \le r^s$ for all $x \in \R^n$ and $r > 0$.

(i) If $\mathcal{H}^s(E)>0$, then Frostman's lemma 
provides a measure $\mu$ with density
$\mu(B(x_0,r))\leq r^s$ for all $B(x_0,r)\subset\R^n$. Then Lemma
\ref{the:density} yields that $\mu$ satisfies the condition \eqref{eq:ehto}.
The other claim follows similarly
from \eqref{eq:natural} and Lemma \ref{the:density}.

(ii) If the measure $\mu$ satisfies the condition
\eqref{eq:ehto}, then Lemma \ref{the:density} implies that
$\mu(B(x_0,r)) \leq C r^s$ for all $B(x_0,2r) \subset \O$, and again Frostman's
lemma indicates that $\dim_\mathcal{H}(\spt \mu )\geq s$; especially $\dim_\mathcal{H}(E) \ge s$.
\end{proof}

\section{Existence of the first eigenfunction} \label{sec:eigenex}

\begin{theorem}
Any $(p,\mu)$-eigenvalue $\lambda$ must be strictly positive, 
and there exists a smallest
$\lambda>0$ and a corresponding first eigenfunction $u \in W^{1,p}_0(\O)$ 
solving \eqref{eq}.
\end{theorem}
\begin{proof}
The weak form of \eqref{eq} reads
\begin{equation}
\label{eq:eq_weak}
\int_\O \abs{\nabla u}^{p-2}\nabla u \cdot \nabla \vp \,dx
= \lambda \int_\O \abs{u}^{p-2} u \vp \,d\mu
\end{equation}
for each $\vp \in \cao$ or, equivalently, for each $\vp \in W^{1,p}_0(\O)$.
Assuming a solution $u \in W^{1,p}_0(\O)$ exists for some $\lambda \in \R$,
we let $\vp=u$ in \eqref{eq:eq_weak} and obtain 
\[
\int_\O \abs{\nabla u}^{p}\,dx
= \lambda \int_\O \abs{u}^{p}\,d\mu.
\]
Thus every pair $(\lambda, u)$ of solutions must satisfy the Rayleigh equation
\begin{equation}
\label{eq:rayleigh1}
\lambda = \frac{\int_\O \abs{\nabla u}^{p}\,dx}{\int_\O \abs{u}^{p}\,d\mu}.
\end{equation}
Since $\O$ was assumed bounded and $q > p$, we have, by H\"older's inequality,
that 
\[
 \norm{u}_{p,\mu} \le \mu(\O)^{\frac{1}{p}-\frac{1}{q}}\norm{u}_{q,\mu}
\]
and further by the embedding \eqref{eq:ehto} 
that $\norm{u}_{p,\mu} \le
C\norm{\nabla u}_{p}$, where  $C > 0$ does not depend on $u$. Thus
$\lambda \ge 1/C > 0$ in \eqref{eq:rayleigh1}.

To prove the existence of the first eigenfunction, let
\[
A = \left\{ \frac{\int_\O \abs{\nabla u}^{p}\,dx}{\int_\O \abs{u}^{p}\,d\mu}
 \colon \, u \in W^{1,p}_0(\O)\right\}
\]
and let $\lambda = \inf A > 0$. Take a sequence $u_n \in
W^{1,p}_0(\O)$ such that 
\[
\lim_{n\to\infty}\frac{\int_\O \abs{\nabla u_n}^{p}\,dx}{\int_\O \abs{u_n}^{p}\,d\mu}=\lambda.
\]
In order to proceed we assert that the embedding of
$W^{1,p}_0(\O)$ to $L^p(\O,\mu)$ is compact.
From \eqref{eq:mitta}
and \cite[Thm. 11.9.1/3]{Mazja} one deduces that
for each bounded sequence in $W^{1,p}_0(\O)$
there exist a function
$v$ and a subsequence $v_n$ such that $v_n \to v$ in $L^p(\Omega',\mu)$ for
each $\Omega'\subset\subset \O$. Now, the estimate
\[
\norm{v_n - v}^p_{L^p(\O,\mu)} \le
\norm{v_n - v}^p_{L^p(\O',\mu)} +
\norm{v_n - v}^p_{L^q(\O,\mu)}\mu(\O\setminus\O')^{1-\frac{p}{q}}
\]
yields the asserted compactness. Hence there exists a function
$u \in W^{1,p}_0(\O)$ such that
\[
R(u) := \frac{\int_\O \abs{\nabla u}^{p}\,dx}{\int_\O \abs{u}^{p}\,d\mu}=\lambda.
\]
Now $u$ is a weak solution to \eqref{eq}, because
\[
  \lim_{t \to 0}\frac{R(u+t\varphi) - R(u)}{t} = 0
\]
for each $\varphi \in C^\infty_0(\O)$, which yields
\[
  \frac{1}{\int_\O |u|^p\,d\mu}\left(\lambda \int_\O |u|^{p-2} u\varphi\,d\mu
- \int_\O |\nabla u|^{p-2}\nabla u \cdot \nabla \varphi\,dx\right) = 0.
\]
Finally by scaling, $\int_\O |u|^p\,d\mu = 1$ may be assumed without loss of generality, which completes the proof.
\end{proof}

\section{Continuity of eigenfunctions} \label{sec:continuity}

We now set out to prove our main result,
a priori H\"older continuity of any
$(p,\mu)$-eigenfunction in the case where $p<n$.
The first step is a Moser iteration that employs
a result by Mal\'y and Ziemer \cite[Cor. 1.95]{MalyZiemer}.
\begin{theorem}\label{thm:sup}
Let $u \in W^{1,p}_0(\O)$ be a $(p,\mu)$-eigenfunction in $\O$, assume
$p<n$, and let
$B(x_0,2r)\subset \O$. Then for any $0<\sigma < 1$,
\begin{equation}\label{eq:sup}
\sup_{B(x_0,\sigma r)} u^+
\leq \frac{C}{(1-\sigma)^{n/p}}
      \l( \vint_{B(x_0,r)}(u^+)^p\, dx  \r)^{1/p},
\end{equation}
where $C= C(p,n,q,\lambda)$. The same estimate holds for $u^-$.
\end{theorem}

\begin{proof}
We first consider the case $\sigma = 1/2$.  
Let $s \geq 1$ and denote $\beta = p(s-1)+1$. Let $\eta \in C_0^\infty(B(x_0,r))$,
$\eta \geq 0$.
Test the equation \eqref{eq} with
$\vf = \eta^p (u^+)^\beta \in W^{1,p}_0(\O)$
to obtain
\begin{align*}
&\beta \int_\O \eta^p (u^+)^{\beta -1} |\nabla u|^p\, dx
+ p\int_\O \eta^{p-1}(u^+)^\beta |\nabla u|^{p-2}\nabla u \cdot \nabla \eta \, dx\\
&= \lambda \int_\O \eta^p |u|^{p-1}(u^+)^\beta \, d\mu.
\end{align*}
Denoting $v= (u^+)^s$, the previous equality reads
\begin{align*}
\int_\O  |\eta\nabla v|^p\, dx
&\leq \frac{sp}{\beta}\int_\O  |\eta \nabla v|^{p-1}|v\nabla \eta|\, dx 
+ \frac{\lambda s^p}{\beta} \int_\O \eta^p v^p \, d\mu\\
&\leq p\int_\O  |\eta \nabla v|^{p-1}|v\nabla \eta|\, dx 
+ \lambda s^{p-1} \int_\O \eta^p v^p \, d\mu,
\end{align*}
and by Young's inequality,
\begin{align*}
  (1-(p-1)\delta^{p'}) \int_\O |\eta \nabla v|^p\, dx
\leq \delta^{-p}\int_\O |v\nabla \eta|^p\, dx
+ \lambda s^{p-1} \int_\O \eta^p v^p \, d\mu.
\end{align*}
Choose $\delta$ so that $1-(p-1)\delta^{p'}= \frac{1}{2}$. Then
\begin{equation}\label{eq:A}
  \begin{aligned}
 \int_\O |\eta \nabla v|^p \, dx
&\leq C \int_\O |v \nabla \eta|^p \, dx 
+ 2 \lambda s^{p-1} \int_\O \eta^p v^p \, d\mu\\
&\leq C s^p\l( \int_\O |v \nabla \eta|^p \, dx 
+  \lambda \int_\O \eta^p v^p \, d\mu  \r).
\end{aligned}
\end{equation}
The last term is estimated using \cite[Cor. 1.95]{MalyZiemer}:
\begin{align}
\label{eq:B}
\int_\O \eta^p v^p \, d\mu
\leq C_{n,p,q} \l(  
   \gamma \int_\O |\nabla(\eta v)|^p\, dx 
   + \gamma^{1-p/\vep}  \int_\O (\eta v)^p\, dx\r),
\end{align}
for any $\gamma > 0$, where $ \vep = \min\{(n-p)(\frac{q}{p}-1),n(p-1)\}$ by \eqref{eq:mitta}.
Substituting \eqref{eq:B} to \eqref{eq:A} we arrive at
\begin{equation*}
(1-C_1s^p \gamma) \int_\O |\eta\nabla v|^p \, dx
\leq C_1 s^p (1+\gamma) \int_\O |v\nabla \eta|^p\, dx
 + C_1 s^p \gamma^{1-p/\vep} \int_\O |\eta v|^p\, dx  .
\end{equation*}
Choose $\gamma = 1/( 2 C_1 s^p)$ to obtain
\begin{equation*}
\int_\O |\eta \nabla v|^p\, dx
\leq (2 C_1 s^p + 1) \int_\O |v\nabla\eta|^p \, dx
+ C s^{(p^2)/\vep} \int_\O |\eta v|^p\, dx. 
\end{equation*}
Rearranging the constants yields
\begin{equation}
 \norm{\eta \nabla v}_p \leq C s^{p/\vep} \norm{v(\eta + |\nabla \eta|)}_p,
\end{equation}
where $C= C(p,n,q,\lambda)$.
By Sobolev's inequality and the fact that $s \geq 1$ we have 
\begin{equation*}
\norm{\eta v}_{p\chi}\leq C s^{p/\vep} \norm{v(\eta + |\nabla \eta|)}_p\,,
\end{equation*}
where $ \chi = n/(n-p)$. Pick two numbers $h$ and $h'$ such that 
$0 < h' < h \leq r$. 
Choose $\eta \in C_0^\infty(B(x_0,h)) $ so that $ \eta = 1$ in $B(x_0,h')$,
$0\leq \eta \leq 1$ in $B(x_0,h)$ and $|\nabla \eta| \leq C(h-h')^{-1}$. Then
\begin{equation*}
  \norm v_{p\chi; B(x_0,h')}\leq C s^{p/\vep} (h-h')^{-1}\norm v_{p;B(x_0,h)} .
\end{equation*}
Denoting $\alpha = ps$, this reads
\begin{equation}
\label{eq:iterate1}
\norm {u^+}_{\alpha \chi ;B(x_0,h')}
\leq \l[C(\alpha/p)^{p/\vep} (h-h')^{-1}  \r]^{p/\alpha} \norm{u^+}_{\alpha;B(x_0,h)}.  
\end{equation}
Note that when $s=1$ (i.e. $\alpha = p$), the right-hand side is finite.
Next we set up an iteration scheme.
Denote $\alpha_i = \chi^ip$, $h_i= r(1/2 + 2^{-i-1})$ and $h_i' = h_{i+1}$for $i \ge 0$,
so that \eqref{eq:iterate1} reads
\begin{equation*}
\begin{aligned}
\norm{u^+}_{p\chi^{i+1};B(x_0,h_{i+1})}
\leq \l[C \chi^{ip/\vep} (h_i-h_{i+1})^{-1}\r]^{1/\chi^i} \norm{u^+}_{p\chi^i;B(x_0,h_i)}.
\end{aligned}
\end{equation*}
Take averages on both sides (denoting $\abs{B(0,1)} = \omega_n$) to obtain
\begin{equation}\label{iteroitava}
\begin{aligned}
&\l(\vint_{B(x_0,h_{i+1})}|u^+|^{p\chi^{i+1}} dx \r)^\frac{1}{p\chi^{i+1}}\\
&\leq \omega_n^{1/(n\chi^i)} \l[\frac{h_i}{h_{i+1}^{1/\chi}} \r]^\frac{n}{p\chi^i}
    \l[C \chi^{ip/\vep} (h_i-h_{i+1})^{-1}\r]^{1/\chi^i} 
    \l(\vint_{B(x_0,h_i)}|u^+|^{p\chi^i} dx \r)^\frac{1}{p\chi^i}\\
&\leq \l[C 2^{i+2}\chi^{ip/\vep}\r]^{1/\chi^i}
    r^{\frac{n}{p\chi^i}-\frac{n}{p\chi^{i+1}}- \frac{1}{\chi^i}}
     \l(\vint_{B(x_0,h_i)}|u^+|^{p\chi^i} dx \r)^\frac{1}{p\chi^i}\\
&= \l[C 2^{i+2}\chi^{ip/\vep}\r]^{1/\chi^i}
     \l(\vint_{B(x_0,h_i)}|u^+|^{p\chi^i} dx \r)^\frac{1}{p\chi^i}.
\end{aligned}
\end{equation}
Now iterating \eqref{iteroitava} yields
\begin{equation*}
\begin{aligned}
&||u^+||_{\infty; B(x_0,r/2)} 
=  \lim_{l\to \infty}\l(\vint_{B(x_0,1/2)} |u^+|^l\, dx \r)^{1/l}\\
&= \lim_{m\to \infty}\l(\vint_{B(x_0,h_m)} |u^+|^{p\chi^m}\, dx \r)^\frac{1}{p\chi^m}\\
&\leq \lim_{m\to \infty} 
  C^{\sum_{i=1}^{m-1} \frac{1}{\chi^i}} (2\chi^{p/\vep})^{\sum_{i=1}^{m-1}\frac{i}{\chi^i}}
  \l(\vint_{B(x_0,r)}|u^+|^p\,dx\r)^{1/p}\\
&= C(n,p,q,\lambda) \l(\vint_{B(x_0,r)}|u^+|^p\,dx\r)^{1/p},
\end{aligned}
\end{equation*}
completing the case $\sigma=1/2$.

If $\sigma \in (0,1/2)$, then
$B(x_0,\sigma r)\subset B(x_0,r/2)$, and we have
\begin{equation*}
\begin{aligned}
\norm{u^+}_{\infty;B(x_0,\sigma r)}
&\leq \norm{u^+}_{\infty;B(x_0,r/2)}
\leq C r^{-n/p} \norm{u^+}_{p;B(x_0,r)} \\
&\leq  C r^{-n/p} (1-\sigma)^{-n/p} \norm{u^+}_{p;B(x_0,r)}.
\end{aligned}
\end{equation*}
On the other hand, if $\sigma \in (1/2, 1)$, then for each ball 
$B(z,(1-\sigma)r)\subset B(x_0,\sigma r)$ we have
$B(z, 2(1-\sigma)r)\subset B(x_0,r)$ and hence
\begin{equation*}
\begin{aligned}
\norm{u^+}_{\infty;B(z,(1-\sigma)r)}
&\leq C r^{-n/p}(1-\sigma)^{-n/p}\norm{u^+}_{p;B(z,2(1-\sigma)r)}\\
&\leq   C r^{-n/p}(1-\sigma)^{-n/p}\norm{u^+}_{p;B(x_0,r)}.
\end{aligned}
\end{equation*}
This implies that
\begin{equation*}
\norm{u^+}_{\infty;B(x_0,\sigma r)}
\leq  C r^{-n/p}(1-\sigma)^{-n/p}\norm{u^+}_{p;B(x_0,r)}.
\end{equation*}
\end{proof}

To proceed from boundedness to continuity, we extend below a result by
Kilpel\"ainen and Zhong \cite[Thm. 1.14]{KilpZhong}. Their result
is stated for positive measures only, but the proof holds almost
verbatim for signed measures as well.
Recently the same result has been obtained by Kuusi and Mingione
\cite{Tuomo_Rosario} using a different approach.

\begin{theorem}\label{thm:TeroXiao}
Let $u\in W_0^{1,p}(\O)$ be a weak solution of
$$
-\Delta_p u = \mu \ \text{ in } \O,
$$  
where $\mu$ is a signed Radon measure in $\O$ such that there are constants
$M>0$ and $0< \alpha < 1$ with
\begin{equation}\label{eq:merkkimitta}
|\mu|(B(x_0,R))\leq MR^{n-p +\alpha(p-1)}  
\end{equation}
whenever $B(x_0,4R) \subset \O$. Then $u\in C^{0,\alpha}(\O)$.
\end{theorem}

\begin{proof}
Fix $B(x_0,R) \subset \O$ so that $B(x_0, 4R) \subset \O$, and let $h$ be the
$p$-harmonic function in $B(x_0,R)$ with
$u-h\in W_0^{1,p}(B(x_0,R))$.
Then, for each $0< r \le R$,
\begin{equation}\label{eq:vertailu}
\begin{aligned}
&\int_{B(x_0,r)} |\nabla u|^p\, dx\\
&= \int_{B(x_0,r)} (|\nabla u|^{p-2}\nabla u -|\nabla h|^{p-2}\nabla h)
       \cdot (\nabla u-\nabla h)\, dx\\
&\qquad+ \int_{B(x_0,r)} |\nabla h|^{p-2}\nabla h \cdot (\nabla u-\nabla h)\, dx
+ \int_{B(x_0,r)} |\nabla u|^{p-2}\nabla u \cdot \nabla h\, dx\\ 
&\leq \int_{B(x_0,R)}    (|\nabla u|^{p-2}\nabla u -|\nabla h|^{p-2}\nabla h)
       \cdot (\nabla u-\nabla h)\, dx\\
& \qquad + \int_{B(x_0,r)} |\nabla h|^{p-1}|\nabla u|+ |\nabla h||\nabla u|^{p-1}\, dx.  
\end{aligned}
\end{equation}
By Adams' inequality  \cite[Cor. 1.93]{MalyZiemer} and Young's inequality
we have, for each $\varepsilon > 0$,
\begin{equation}
\label{eq:intermediate1}
\begin{aligned}
& \int_{B(x_0,R)}    (|\nabla u|^{p-2}\nabla u -|\nabla h|^{p-2}\nabla h)
       \cdot (\nabla u-\nabla h)\, dx\\
&= \int_{B(x_0,R)} (u-h)\,d\mu 
\leq \int_{B(x_0,R)}|u-h|\, d|\mu|   \\
&\leq C R^{(p-1)(n-p+\alpha p)/p} (\int_{B(x_0,R)} |\nabla u - \nabla h|^p\, dx )^{1/p}\\  
&\leq  C_\varepsilon R^{n-p+\alpha p} 
  +\frac{\vep}{2} \int_{B(x_0,R)} |\nabla u|^p\, dx,
\end{aligned}
\end{equation}
where we also used the minimizing property of $p$-harmonic functions.
Next recall that
 if $h$ is $p$-harmonic in $\O$, then
\begin{equation*}
\vint_{B(x_0,r)} |\nabla h|^p \, dx 
\leq C 
\vint_{B(x_0,R)} |\nabla h|^p\, dx  
\end{equation*}
for each $0< r< R$ with $B(x_0,R) \subset \O$ (see e.g.
\cite[Lemma 2.1]{Kilp}).
We employ this (along with Young's inequality and the minimizing property)
to estimate the last term on the right hand side of \eqref{eq:vertailu}:
\begin{equation}
\label{eq:intermediate2}
\begin{aligned}
&\int_{B(x_0,r)} |\nabla h|^{p-1}|\nabla u|+ |\nabla h||\nabla u|^{p-1}\, dx\\   
&\leq \frac{1}{2}\int_{B(x_0,r)}|\nabla u|^p\, dx 
  + C\int_{B(x_0,r)} |\nabla h|^p\, dx\\      
&\leq \frac{1}{2}\int_{B(x_0,r)} |\nabla u|^p\, dx 
  + C\l(\frac{r}{R}\r)^n  \int_{B(x_0,R)} |\nabla h|^p \, dx  \\
&\leq \frac{1}{2}\int_{B(x_0,r)} |\nabla u|^p\, dx 
  + C\l(\frac{r}{R}\r)^n  \int_{B(x_0,R)} |\nabla u|^p \, dx.
\end{aligned}
\end{equation}
When substituting \eqref{eq:intermediate1} and \eqref{eq:intermediate2} to
\eqref{eq:vertailu} we obtain
\begin{equation*}
\begin{aligned}
 \int_{B(x_0,r)} |\nabla u|^p\, dx 
\leq C_\varepsilon R^{n-p+\alpha p}
  + \l(C\l(\frac{r}{R}\r)^n +\vep\r) \int_{B(x_0,R)} |\nabla u|^p \, dx.
\end{aligned}
\end{equation*}
We are now in position to appeal to \cite[Lemma III.2.1]{G} which yields
\begin{equation*}
 \int_{B(x_0,r)}|\nabla u|^p\, dx \leq C\l(\frac{r}{R}\r)^{n-p+ p\alpha } 
\end{equation*}
for $r< R$, if $B(x_0,4R) \subset \O$.
Thus $u\in C^{0,\alpha}(\O)$ by the
Dirichlet growth theorem \cite[Thm. III.1.1]{G}.
\end{proof}

We now obtain H\"older continuity as a direct consequence of
Theorems \ref{thm:sup} and \ref{thm:TeroXiao}:
\begin{theorem}\label{thm:jvuus}
Let $u \in W^{1,p}_0(\O)$ be a $(p,\mu)$-eigenfunction, and assume
$p<n$. Then
$u \in C^{0,\alpha}(\O)$ with any $0 < \alpha < 1$ satisfying
\[
\alpha \le \frac{q-p}{p(p-1)}(n-p).
\]
\end{theorem}
\begin{proof}
By the definition of the operator $\Delta_{p,\mu}$ we have
$$
-\Delta_p u= \lambda |u|^{p-2}u\mu.
$$
Using Theorem \ref{thm:sup} and Lemma \ref{the:density} we estimate the measure
$|u|^{p-2}u\mu$: if $B(x_0,2r)\subset\O$, then
\begin{equation*}
\begin{split}
 |u^{p-2}u\mu|(B(x_0,r)) & = \int_{B(x_0,r)}|u|^{p-1}\, d\mu 
\leq \norm{u^{p-1}}_{\infty,B(x_0,r)}\, \mu(B(x_0,r)) \\
& \leq C r^{q(n-p)/p} = C r^{n-p + \alpha(p-1)},
\end{split}
\end{equation*}
where $C= C(n,p,q,\lambda,u)$. The claim now follows from Theorem
\ref{thm:TeroXiao}.
\end{proof}

\section{Simplicity of the first eigenvalue} \label{sec:first}
\begin{theorem}\label{merkki}
The first $(p,\mu)$-eigenfunction $u_1$ does not change sign.
\end{theorem}
\begin{proof}
We notice that $|u_1|$ is a minimizer of the Rayleigh quotient and
therefore 
also $|u_1|$ is a $(p,\mu)$-eigenfunction. 
By definition we have
$$
 \int_\O |\nabla u_1|^{p-2}\nabla |u_1| \cdot \nabla \vf\, dx
= \int_\O  |u_1|^{p-1} \vf\, d\mu \geq 0
$$
for all test functions $\varphi$, hence $|u_1|$ is a supersolution to
the $p$-Laplace equation. By the weak Harnack inequality $|u_1|$ is locally
bounded away from zero. Let $B \subset \Omega$ be a ball; we have
$|u_1| \ge \delta > 0$ in $B$. By the lattice property of Sobolev spaces 
the function $v = \operatorname{sgn} (u_1)$ belongs to $W^{1,p}(B)$.
Further, since functions in $W^{1,p}(B)$ have representatives that are
absolutely continuous on almost every line parallel to the coordinate axes,
we obtain $\nabla v = 0$ a.e. in $B$. From the Poincaré inequality we
infer that $v = v_B$, so $v$ cannot change sign in $B$. Therefore $u_1$
cannot change sign in $B$, and since $B \subset \O$ was arbitrary, the proof is completed.
\end{proof}

\begin{theorem}
\label{the:sign}
The first $(p,\mu)$-eigenvalue is simple.
\end{theorem}
\begin{proof}
We use a method devised by Belloni and Kawohl \cite{BK}.
By homogeneity we can assume that $\int_\O |u|^p\, d\mu = 1$
for the first eigenfunction, so for the first eigenvalue $\lambda_1$ we have
\begin{equation} \label{eq:rayleigh}
\lambda_1
= \inf \l\{\int_\O |\nabla v |^p \, dx \colon  v\in 
W^{1,p}(\O),\ \norm v_{L^p(\O,\mu)}= 1  \r\}.
\end{equation}
Assume there are two minimizers  $u, v$ of  \eqref{eq:rayleigh}. The minimizers
can be assumed positive by Theorem \ref{merkki}. The function $w= \eta^{1/p}$ with $\eta= (u^p + v^p)/2$
is also admissible in \eqref{eq:rayleigh},
because 
$$
\int_\O |w|^p \, d\mu 
= \frac{1}{2} \l(\int_\O u^p + v^p \, d\mu \r)= 1.
$$
Now 
$\nabla w = \frac{1}{2}\eta^{-1+1/p}\, (u^{p-1}\nabla u + v^{p-1}\nabla v)$,
and by convexity we obtain, denoting $s= u^p(u^p + v^p)^{-1}\in (0,1)$,
\begin{equation}
\label{eq:belloni}
\begin{split}
|\nabla w|^p
&= \eta^{1-p} \l| \frac{1}{2}(u^{p-1}\nabla u + v^{p-1}\nabla v)\r|^p\\
&= \eta \l| \frac{1}{2} 
    \l( \frac{u^p}{\eta}\frac{\nabla u}{u} + \frac{v^p}{\eta}\frac{\nabla v}{v} 
    \r) \r|^p\\
&= \eta \l| 
         s\frac{\nabla u}{u} + (1-s)\frac{\nabla v}{v}
        \r|^p \\
&\leq \eta \l( 
         s\l|\frac{\nabla u}{u}\r|^p + (1-s)\l|\frac{\nabla v}{v}\r|^p
           \r)\\
&= \frac{1}{2} \l( 
         u^p\l|\frac{\nabla u}{u}\r|^p + v^p\l|\frac{\nabla v}{v}\r|^p
           \r)
= \frac{1}{2} \l(|\nabla u|^p + |\nabla v|^p \r).
\end{split}
\end{equation}
Hence 
\begin{equation}  \label{eq:mini}
\int_\O |\nabla w|^p \,dx
\leq \frac{1}{2} \l(\int_\O|\nabla u|^p \, dx + \int_\O |\nabla v|^p\, dx \r).
\end{equation}
Because $u$ and $v$ are both minimizers of \eqref{eq:rayleigh}, equality
holds in \eqref{eq:mini}, and therefore also in \eqref{eq:belloni} for
a.e. $x\in\O$.  Since the function $|\cdot|^p$ is strictly convex, we must have
\[
\frac{\nabla u}{u} = \frac{\nabla v}{v} \ \text{ a.e,}
\]
i.e.
\[
u\nabla v - v \nabla u = 0 \ \text{ a.e,}
\]
i.e. ($v \ne 0$ by definition)
\[
\nabla\left(\frac{u}{v}\right) = 0 \ \text{ a.e.}
\]
Therefore $u(x) = \text{const}\cdot v(x)$ for a.e. $x\in \O$.
\end{proof}

\section{A counterexample in the case $p=n$}
\label{sec:n}
\begin{theorem}
When $p=n$, there exists a measure $\mu$ on $\O$ satisfying \eqref{eq:ehto}
such that any first $(n,\mu)$-eigenfunction is not $\alpha$-Hölder continuous
for any $\alpha \in (0,1)$.
\end{theorem}
\begin{proof}
Recall first from \cite[Cor. 11.8.1]{Mazja}
that when $p=n$, the embedding \eqref{eq:ehto} is equivalent to
\begin{equation}\label{eq:n}
\mu(B(x,r)) \leq C |\log r|^{-q(n-1)/n}
\end{equation}
for all $x \in \O$ and all small radii $r$.
On the other hand, assuming that an $\alpha$-Hölder continuous first
$(n,\mu)$-eigenfunction $u$ exists for some measure $\mu$ and for some
$\alpha \in (0,1)$, it follows from 
\cite[Rmk. 2.7]{Kilp} that 
\begin{equation*}
  |u|^{n-1}\mu(B(x,r)) \le C r^{\alpha(n-1)}
\end{equation*}
for all $x \in \O$ and all small radii $r$. Moreover, this first eigenfunction
satisfies $|u|>0$ in $\O$ by Theorem \ref{merkki}, so for any nonempty
$E \subset \O$ there exists a point $x_0 \in E$ and a small radius $r$
such that $|u| > \delta > 0$ in $B(x_0,r)$. Hence the growth condition
\begin{equation}\label{n_kasvu}
  \mu(B(x_0,r)) \le C r^{\alpha(n-1)}
\end{equation}
should hold for all small radii.

We reach a contradiction by constructing a set $E \subset \O$
and a measure
$\mu$ supported on $E$ such that \eqref{eq:n} holds for all points
$x \in \Omega$, while for all $x \in E$ there exist arbitrarily small
radii $r$ for which the growth condition \eqref{n_kasvu}
fails for each $\alpha \in (0,1)$.

Let $h(r) = \abs{\log r}^{-q(n-1)/n}$, and choose $r_0 > 0$
such that $B^0_1 = \overline{B}(x_0, r_0/2) \subset \O$. Define $r_k$ inductively by
\[
 h(r_{k+1}) = \frac{1}{2}h(r_k).
\]
Choose closed balls $B^1_1, B^1_2 \subset B^0_1$ having empty intersection
and diameter $r_1$. Similarly at stage $k$ choose inside each ball
$B^k_1,\ldots,B^k_{2^k}$ two non-overlapping balls having diameter $r_{k+1}$.

Let $E = \bigcap_k \bigcup_j B^k_j$.
Define a set function $\widetilde{\mu}$ on the collection 
$\mathcal{B} = \{B^k_j\}$ by
$\widetilde{\mu}(B^k_j) = h(r_k)$, and define a measure $\mu$ on $\Omega$ by
\[
 \mu(A) = \inf\left\{\sum_{i=1}^\infty \widetilde{\mu}(B_i) \, \colon \, A \cap E \subset
\bigcup_{i=1}^\infty B_i,\  B_i \in \mathcal{B} \right\}.
\]
Then $\mu$ satisfies
\[
 \mu(B(x,r)) \le Ch(r)
\]
for each $x \in \O$ and each $0 < r < r_0$, and also
\[
 \mu(B(x,r)) \ge ch(r)
\]
for each $x \in E$ and each $0 < r < r_0$.
\end{proof}

\section{Conclusions}
\label{sec:conclusions}
We introduced a way to define the $p$-Laplace operator with respect to
measures satisfying an Adams type embedding condition. The definition provided
unique solutions to Poisson problems with zero boundary data, and the class
of admissible measures included natural measures for self-similar sets whose
Hausdorff dimension depended on the embedding parameters.
The main part of our analysis was devoted to the $(p,\mu)$-eigenvalue problem.
We proved positivity of eigenvalues, existence of the first eigenfunction, and
a priori H\"older continuity of eigenfunctions. We also showed that the
first eigenvalue is simple and that the first eigenfunction does not change
sign. Finally we noted that the case $p=n$ does not necessarily yield
Hölder continuous eigenfunctions.

An open problem that we find particularly interesting is to find a sharp
embedding condition for the measure that guarantees continuity of
eigenfunctions in the case $p=n$.
Our treatment also leaves room for some generalizations. For example, what
can be said when the boundary values are nonzero?
One might also consider more general nonlinear operators with $p$-growth
in the measure setting.

While this article was under review, the authors were informed about
articles \cite{BelRou} and \cite{FranLam} where related 
problems are studied.

\section*{Acknowledgements} We thank Tero Kilpeläinen,
Juha Kinnunen and Ville Suomala for helpful discussions, and Juan Manfredi
for informing us about the article \cite{Hu}. We also thank
the anonymous referee for many helpful and clarifying remarks.

\appendix

\section{The natural measure of a self-similar set} \label{sec:fractals}

In this appendix we recall (from \cite{Falc1}, \cite{Falc2}) the definition
of the natural measure for a self-similar set, used in the proof of
Theorem \ref{thm:hausdorff}.

A function $f:\O\to \O$ is a \emph{similarity transformation} 
with ratio $r$ if $|f(x)-f(y)| = r|x-y|$ for all $x,y \in \O$ 
and for some $0 < r < 1$. A \emph{self-similar set} related to finitely many
similarities
$f_1, \dots, f_N$ is the unique non-empty compact set $E \subset \O$ such that
$$
E = f_1(E)\cup \dots \cup f_N(E).
$$

A self-similar set $E \subset \O$ supports \emph{self-similar measures}:
For each similarity $f_i$ assign a probability $p_i$, $0 \leq  p_1 \leq 1$ ,
such that $\sum_{i=1}^N p_i= 1$. Then there exists a unique Borel probability
measure $\mu$ on $\O$ such that
$$
\mu(A)= \sum_{i=1}^N p_i \mu(f_i^{-1}(A))
$$
for all Borel sets $A \subset \O$, see \cite[Thm. 2.8]{Falc2}.

A self-similar set $E \subset \O$ related to similarities $f_1,\ldots,f_N$
satisfies the \emph{open set condition} if
there exists a non-empty open set $U \subset \O$ such that 
\begin{equation}\label{eq:open}
  \bigcup_{i=1}^N f_i(U)\subset U \quad
\text{  and  }\quad f_i(U)\cap f_j(U)= \emptyset \text{ for } i \not= j.
\end{equation}

If a self-similar set $E \subset \O$ satisfying the open set condition 
is the invariant set of similarities $f_i$ with ratios $r_i$,
then $E$ has Hausdorff dimension $s$ satisfying
$$
\sum_{i=1}^N r_i^s = 1
$$
and further $E$ has positive and finite $\mathcal{H}^s$-measure.
Moreover, the self-similar measure $\mu_E$ with probabilities
$p_i = r_i^s$ is evenly destributed and there is a constant $C$ such that
\begin{equation}\label{eq:natural}
 \mu_E(B(x,r)) \leq C r^s
\end{equation}
for all balls $B(x,r)\subset \O$. See \cite[Thm. 9.3 ]{Falc1}. 
We call the measure $\mu_E$  the \emph{natural measure} on $E$.

\end{document}